\documentclass[11pt,a4paper]{amsart}
\usepackage[english]{babel}
\usepackage[utf8]{inputenc}  
\usepackage[T1]{fontenc}  
\usepackage{amscd,amssymb,amsmath,latexsym,mathrsfs,upgreek,graphicx,marginnote}
\usepackage[all]{xy}
\usepackage[lowtilde]{url}
\usepackage{hyperref} \hypersetup{breaklinks=true}

\def\ov#1{{\overline{#1}}}

\def\wh#1{{\widehat{#1}}}

\newcommand{\longhookrightarrow}{\lhook\joinrel\longrightarrow}

\newcommand{\codim}{{\operatorname{codim}}}

\newcommand{\h}{{\operatorname{h}}}

\newcommand{\Gm}{{\mathbb G}_{\rm m}}

\newcommand{\cyc}{{\operatorname{cyc}}}
\newcommand{\ct}{{\operatorname{ct}}}
\renewcommand{\and}{{\quad \text{ and } \quad }}

\newcommand{\im}{{\operatorname{im}}}
\renewcommand{\Im}{{\operatorname{im}}}

\newcommand{\Adj}{{\operatorname{adj}}}

\newcommand{\Spec}{{\operatorname{Spec}}}

 \newcommand{\C}{{\mathbb{C}}}
 
\newcommand{\G}{{\mathbb{G}}} \newcommand{\K}{{\mathbb{K}}} 
 \newcommand{\N}{{\mathbb{N}}} 
\renewcommand{\P}{{\mathbb{P}}} \newcommand{\Q}{{\mathbb{Q}}} 
  
\newcommand{\Z}{{\mathbb{Z}}}

\newcommand{\cE}{{\mathcal{E}}}

\newcommand{\cO}{{\mathcal{O}}} \newcommand{\cP}{{\mathcal{P}}}

\newcommand{\bfa}{{\boldsymbol{a}}}
\newcommand{\bfb}{{\boldsymbol{b}}}
\newcommand{\bfc}{{\boldsymbol{c}}}
\newcommand{\bfd}{{\boldsymbol{d}}}
\newcommand{\bfe}{{\boldsymbol{e}}}

\newcommand{\bft}{{\boldsymbol{t}}}
\newcommand{\bfu}{{\boldsymbol{u}}}

\newcommand{\bfx}{{\boldsymbol{x}}}
\newcommand{\bfy}{{\boldsymbol{y}}}

\newcommand{\bfzero}{{\boldsymbol{0}}}

\newcounter{thm}
\numberwithin{equation}{section}
\numberwithin{thm}{section}

\theoremstyle{definition}
\newtheorem{definition}[thm]{Definition}

\newtheorem{remark}[thm]{Remark}
\newtheorem{example}[thm]{Example}

\theoremstyle{plain}
\newtheorem{lemma}[thm]{Lemma}
\newtheorem{proposition}[thm]{Proposition}
\newtheorem{theorem}[thm]{Theorem}
\newtheorem*{theorem*}{Theorem}
\newtheorem{corollary}[thm]{Corollary}
\newtheorem{conjecture}[thm]{Conjecture}

\newtheorem{prop-def}[thm]{Proposition-Definition}

\newcommand{\xg}{{\boldsymbol{x}}}

\newcommand{\ag}{{\boldsymbol{a}}}
\newcommand{\bg}{{\boldsymbol{b}}}
\newcommand{\cg}{{\boldsymbol{c}}}

\usepackage{color}
\newcommand\red{\color{red}}

\renewcommand\red{}

\begin{document}

\title[Factorization of bivariate sparse polynomials]{Factorization of bivariate sparse polynomials}
  
\author[Amoroso]{Francesco Amoroso} 
\address{Laboratoire de
  math\'ematiques Nicolas Oresme, CNRS UMR 6139, Universit\'e de
  Caen. BP 5186, 14032 Caen Cedex, France}
\email{francesco.amoroso@unicaen.fr}
\urladdr{\url{http://www.math.unicaen.fr/~amoroso/}}

\author[Sombra]{Mart{\'\i}n~Sombra} \address{Instituci\'o Catalana de
  Recerca i Estudis Avançats (ICREA). Passeig Llu\'is Companys~23,
  08010 Barcelona, Spain \vspace*{-2.5mm}} 
\address{Departament de
  Matem\`atiques i Inform\`atica, Universitat de Barcelona (UB). Gran
  Via 585, 08007 Bar\-ce\-lo\-na, Spain} 
\email{sombra@ub.edu}
\urladdr{\url{http://www.maia.ub.edu/~sombra}}

\date{12/9/2018} 

\thanks{Amoroso was partially supported by the CNRS research project
  PICS 6381 ``Diophantine geometry and computer algebra''. Sombra was
  partially supported by the MINECO research project MTM2015-65361-P}

\begin{abstract}
  We prove a function field analogue of a conjecture of Schinzel on
  the factorization of univariate polynomials over the rationals.  We
  derive from it a finiteness theorem for the irreducible
  factorizations of the bivariate Laurent polynomials in families with
  fixed set of complex coefficients and varying exponents. Roughly speaking,
  this result shows that the truly bivariate irreducible factors of
  these sparse Laurent polynomials, are also sparse.
  {\red The proofs are based on a variant of the toric Bertini's
  theorem due to Zannier and Fuchs, Mantova and Zannier.}
\end{abstract}

\maketitle

\section[Introduction]{Introduction}
\label{sec:introduction}

A polynomial is said to be \emph{sparse} (or \emph{lacunary}) if it
has few terms compared with its degree.  The factorization problem for
sparse polynomials can be vaguely stated as the question of whether
the irreducible factors of a sparse polynomial are also sparse, apart
from obvious exceptions.  Aspects of this problem have been studied in
various settings and for different formalizations of the notion of
sparsenness, see for instance \cite{Lenstra:flp, Schinzel:psrr,
  KaltofenKoiran:fsdfmslpanf,
  AvendanoKrickSombra:fbslp,FilasetaGranvilleSchinzel:igcdasp,
  Grenet:bdflmp, Am-So-Za}. Several of these studies were based on
tools from Diophantine geometry like lower bounds for the height of
points and subvarieties, and unlikely intersections of subvarieties
and subgroups of a torus.

In this text, we consider families of bivariate Laurent polynomials
given as the pullback of a \emph{fixed} regular function on a torus by
a \emph{varying} 2-parameter monomial map. Precisely, let $t,z$ be
variables, $\bfx=(x_{1},\dots,x_{n})$ a set of other $n$ variables,
and
$$F\in \C[\bfx^{\pm1},z^{\pm1}]=\C[x_{1}^{\pm1},\dots,
x_{n}^{\pm1},z^{\pm1}]$$
a Laurent polynomial. For each vector $\bfa=(a_{1},\dots,a_{n})\in \Z^{n}$, we consider the bivariate
Laurent polynomial given as the pullback of $F$ by the monomial map
\begin{math}
  (t,z)\mapsto (t^{\bfa},z)= (t^{a_{1}}, \dots, t^{a_{n}},z)
\end{math}, 
that is 
\begin{equation} \label{eq:11}
F_{\bfa}=  F(t^{\bfa},z)\in \C[t^{\pm1},z^{\pm1}].
\end{equation}
The number of coefficients of each $F_{\bfa}$ is bounded by those of
$F$, and so these Laurent polynomials can be considered as sparse when
$F$ is fixed and $\bfa$ is large.

A first question concerns the irreducibility of $F$. It has been
addressed in~\cite{Zannier:hiaag}, as we next describe.  Let us assume
$F$ irreducible. Under which assumptions $F_{\bfa}$ stays irreducible
for a {\it generic} $\bfa$? Let us consider the following example. The
Laurent polynomial
$F=z^{2}-x_{1}x_{2}^{2}\in \C[x_{1}^{\pm1}, x_{2}^{\pm1},z^{\pm1}]$.
is irreducible. However given
$(a_{1},a_{2})\in \Z^{2}\setminus \{(0,0)\}$ with $a_{1}$ even, the
Laurent polynomial $F_{\bfa}$ is reducible. This show that the sole
assumption that $F$ is irreducible is not enough to get an
irreducibility statement for a generic specialization. 

Zannier's result can be stated as follows. For $\bfa,\bfb\in \Z^{n}$,
we denote by $\langle \bfa,\bfb\rangle=\sum_{i=1}^{n}a_{i}b_{i}$ their
scalar product.

\begin{theorem}[{{\cite[Theorem 3]{Zannier:hiaag}}}]
  Let $F \in \C[\bfx^{\pm 1},z^{\pm 1}]\setminus \C[\bfx^{\pm1}]$ be
  an irreducible Laurent polynomial, that is monic in $z$ and such
  that $F(x_{1}^d,\dots,x_{n}^d,z)$ is irreducible for $d=\deg_z(F)$.
  There is a finite subset
  $\Sigma \subset \Z^{n}\setminus \{\bfzero\}$ such that, for
  each $\bfa\in \Z^{n}$, either there is $\cg\in\Sigma$ with
  $\langle \bfc,\bfa\rangle =0$, or $F(t^\bfa,z)$ is irreducible.
\end{theorem}

{\red As already remarked by the author, the classical Bertini theorem 
may be seen as a version of statement of this shape for ${\Bbb G}^n_{\rm a}$.}

Previously, Schinzel \cite{Sch1} proved a similar result in the same
direction, for Laurent polynomials over $\Q$ satisfying the strong
additional assumption that $F$ is not self-inversive. More recently,
Fuchs, Mantova and Zannier \cite[Addendum to
Theorem~1.5]{FuchsMantovaZannier:fiptvb} showed that the set $\Sigma$
can be chosen independently of the coefficients of $F$.

In the present paper we are  interested in the factorization of
$F_{\bfa}$. Our motivation is an old conjecture of Schinzel
\cite{Schinzel:rppt} on the factorisation of sparse polynomial with
rational coefficients (Conjecture~\ref{Sch0}).  

This conjecture implies the statement below. A Laurent polynomial in
$\Q[\bfx^{\pm 1}]$ is \emph{cyclotomic} if it can be written as a unit
times the composition of a univariate cyclotomic polynomial with a
monomial.

\begin{conjecture}
  Let $F\in \Q[\bfx^{\pm1}]$. There is a finite set of matrices
  $\Omega\subset \Z^{n\times n}$ satisfying the following
  property. For each $\bfa\in\Z^n$, there are $M\in\Omega$ and
  $\bfb\in\Z^{n}$ with $\bfa=M\bfb$ such that if $P$ is an irreducible
  factor of $F(\bfx^M)$, then $P(t^{\bfb})$ is either a product of
  cyclotomic Laurent polynomials, or
  an irreducible factor of $F(t^{\bfa})$.
\end{conjecture}

Our main result in this text is the following function field analogue.

\begin{theorem} 
\label{thm:2}
Let $F\in \C[\bfx^{\pm1},z^{\pm1}]$. There is a finite set of matrices
$\Omega\subset \Z^{n\times n}$ satisfying the following property. For
  each $\bfa\in\Z^n$, there are $M\in\Omega$ and $\bfb\in\Z^{n}$ with
$\bfa=M\bfb$ such that if $P $ is an irreducible factor of
$F(\bfx^M,z)$, then $P(t^{\bfb},z)$ is, as an element of
$\C(t)[z^{\pm1}]$, either a unit or an irreducible factor of
$F(t^{\bfa},z)$.
\end{theorem}

Moreover, we also obtain in Theorem \ref{thm:1} the function field
analogue of Conjecture~\ref{Sch0}.

Theorem \ref{thm:2} shows that for each $\bfa\in \Z^{n}$, there is a
matrix $M$ within the finite set $\Omega \subset \Z^{n\times n}$ and a
vector $\bfb\in \Z^{n}$ with $\bfa=M\bfb$ such that, unless
$F\big(\bfx^M,z\big)=0$, the irreducible factorization
\begin{equation}
  \label{eq:2}
F\big(\bfx^M,z\big)=\prod_{P}P(\bfx,z)^{e_{P}}  
\end{equation}
yields the irreducible factorization in the ring $\C(t)[z^{\pm1}]$
\begin{displaymath}
F_{\bfa}= \gamma\, \prod_{P} {\vphantom\prod}'P(t^{\bfb},z)^{e_{P}},
\end{displaymath}
for $F_{\bfa}$ as in \eqref{eq:11}, the
product being over the irreducible factors $P$ in \eqref{eq:2} such
that $P(t^{\bfb},z)$ is not a unit, and with
$\gamma\in\C(t)[z^{\pm1}]^{\times }$.

Hence, the irreducible factorizations in $\C(t)[z^{\pm1}]$ of the
$F_{\bfa}$'s can be obtained by specializing the irreducible
factorizations of the Laurent polynomials $F(\bfx^M,z)$ for a
\emph{finite} number of matrices~$M$. These irreducible factors of the
$F_{\bfa}$'s are {sparse}, in the sense that they are all represented
as the pullback of a finite number of regular functions on the
$(n+1)$-dimensional torus $\Gm^{n+1}$ by 2-parameter monomial maps.
In particular, both the number of these irreducible factors and of
their coefficients are bounded above independently of $\bfa$.

The proof of Theorem \ref{thm:2} relies on a variant of the
aforementioned result of Zannier. To state it, we first introduce
some further notation.  Let $\bft =(t_1,\ldots,t_{k})$ be a set of
$k$ variables. A matrix $A=(a_{i,j})_{i,j}\in\Z^{n\times k}$ defines
the family of $n$ monomials in the variables $\bft$ given by
$$
\bft^A=\Big(\prod_{j=1}^{k} t_j^{a_{1,j}},\ldots,\prod_{j=1}^k t_j^{a_{n,j}}\Big).
$$

Given $\bfa=(a_{1},\dots, a_{n})\in \Z^{n}$, we can consider it as a
row vector or as a column vector. Thus 
\begin{displaymath}
  \bfx^{\bfa}= \prod_{j=1}^{n}x_{j}^{a_{j}} \and  t^{\bfa}=(t^{a_{1}},\dots, t^{a_{n}}).
\end{displaymath}

\begin{theorem}
\label{BT-poly}
Let $F \in \C[\bfx^{\pm 1},z^{\pm1}]\setminus \C[\bfx^{\pm1}]$ be an
irreducible Laurent polynomial, and $G\in \C[\bfx^{\pm 1}]$ the
coefficient of the term of highest degree in the variable $z$. 
There are finite subsets $\Phi\subset \Z^{n \times n}$ of nonsingular
matrices and $\Sigma \subset \Z^{n}$ of nonzero vectors such that, for
$\bfa\in \Z^{n}$, one of the next alternatives holds:
\begin{enumerate}
\item \label{item:8} there is $\cg\in\Sigma$ such that
  $ \langle  \bfc,\bfa\rangle =0$;
\item \label{item:9} there is $M\in\Phi$ such that $\bfa\in \Im(M) $
  and $F(\xg^M,z)$ is reducible;
\item \label{item:10} the Laurent polynomial
  $F(t^\bfa,z) \in \C[ t^{\pm1},z^{\pm1}]$ is irreducible in
  $\C[t^{\pm1},z^{\pm1}]_{G(t^{\bfa})}$.
\end{enumerate}
\end{theorem}

Back to the factorization problem for sparse polynomials, it is
natural to consider the more general setting of pullbacks of regular
functions on $\Gm^{n}$ by \emph{arbitrary} monomial maps, instead of
only those appearing in \eqref{eq:11}. Let $\bfy=(y_{1},\dots, y_{n})$
and $\bft=(t_{1},\dots, t_{k})$ be groups of $n$ and $k$ variables,
respectively. For a Laurent polynomial $H\in \C[\bfy^{\pm1}]$,
consider the family of $k$-variate Laurent polynomials given by the
pullback of $H$ by the monomial map $\Gm^{k}\to \Gm^{n}$ defined by
$\bft\mapsto \bft^{A}$ for a matrix $A\in \Z^{n\times k}$, that is
\begin{equation*} 
H_{A}=  H(\bft^{A})\in  \C[\bft^{\pm1}] .
\end{equation*}

Denote by $S$ the multiplicative subset of $\C[\bft^{\pm1}]$
generated by the Laurent polynomials of the form $f(\bft^{\bfd})$ for
$f\in \C[z^{\pm1}]$ and $\bfd\in \Z^{k}$.

We propose the following conjecture which, as explained in Remark
\ref{rem:10}, partially generalizes Theorem~\ref{thm:2}.

\begin{conjecture}
  \label{conj:3}
  Let $H\in \C[\bfy^{\pm1}]$ and $k\ge 2$. There is a finite set of
  matrices $\Omega\subset \Z^{n\times n}$ {satisfying the following
    property. For each $A\in\Z^{n\times k}$, there are} $N\in\Omega$ and
  $B\in\Z^{n\times k}$ with $A=N B$ such that if $P $ is an
  irreducible factor of $H(\bfy^N)$, then $P(\bft^{B})$ is, as an
  element of $\C[\bft^{\pm1}]_{S}$, either a unit or an irreducible
  factor of $H(\bft^{A})$.
\end{conjecture}

The validity of this conjecture would imply that the irreducible
factors of the $H_{A}$'s that truly depend on more than one variable,
are also the pullback of a finite number of regular functions on
$\Gm^{n}$ by $k$-parameter monomial maps.  The possible univariate
irreducible factors of the $H_{A}$'s split completely, and so they
cannot be accounted from a finite number of such regular functions.

This conjecture might follow from a suitable toric analogue of the
classical Bertini's theorem that we propose in Conjecture \ref{conj_TBT}.

\medskip \noindent {\bf Plan of the paper.}  In Section
\ref{sec:conj-sche-funct} we state Schinzel's conjecture and our
function field analogue (Theorem \ref{thm:10}).  In
Section~\ref{sec:toric-anal-bert} we recall some facts on fiber
products and prove a variant of the Fuchs-Mantova-Zannier theorem
concerning the irreducibility of pullbacks of cosets by a dominant
maps $W\rightarrow \Gm^n$ (Theorem~\ref{TBT}).  In Section
\ref{sec:pullb-laur-polyn} we prove Theorem \ref{BT-poly}, wereas in
Section \ref{sec:fact-sparse-polyn} we apply this result to prove
Theorem \ref{thm:10} and then Theorem \ref{thm:2}.

\medskip \noindent {\bf Acknowledgments.} We thank Pietro Corvaja,
Qing Liu, Vincenzo Mantova, Juan Carlos Naranjo and Umberto Zannier
for useful conversations. We also thank the anonymous referee for
his/her useful comments. Part of this work was done while the authors
met the Universitat de Barcelona and the Universit\'e de Caen. We thank
these institutions for their hospitality.

\section{A conjecture of Schinzel and its function field analogue} \label{sec:conj-sche-funct}

In~\cite{Schinzel:rppt}, Schinzel proposed the conjecture below on the
factorization of univariate polynomials over $\Q$.



\begin{conjecture}
\label{Sch0}
Let $F\in\Q[\bfx^{\pm1}]$ be a non-cyclotomic irreducible Laurent
polynomial. There are finite sets $\Omega^{0}\subset \Z^{n\times n}$
of nonsingular matrices and $\Gamma\subset \Z^n$ of nonzero vectors
satisfying the following property. Let $\bfa\in\Z^n$; then one of the next
conditions holds:
\begin{enumerate}
\item \label{item:2} there is $\bfc\in\Gamma$ verifying
 $\langle \bfc,\bfa\rangle=0$;
\item \label{item:1} there are $M\in\Omega^{0}$ and $\bfb\in\Z^n$ with
 $\bfa=M\bfb$ such that if 
$$
F(\bfx^M)=\prod_{P}P^{e_{P}}
$$ 
is the irreducible factorization of $F(\bfx^M)$, then
\begin{displaymath}
\frac{F(t^{\bfa})}{\cyc(F(t^{\bfa})}=\prod_{P}\Big(\frac{P(t^{\bfb})}{\cyc(P(t^{\bfb}))}\Big)^{e_{P}}
\end{displaymath}
is the
irreducible factorization of ${F(t^{\bfa})}/{\cyc(F(t^{\bfa})})$.
\end{enumerate}
\end{conjecture}

For the validity of this statement, in its condition \eqref{item:1} it
is necessary to take out the cyclotomic part of $F(t^{\bfa})$ and of
the $P(t^{\bfb})$'s, as shown by the example below.

\begin{example}
 \label{exm:2}
 Set $F=x_{1}+x_{2}-2\in \Q[x_{1}^{\pm1}, x_{2}^{\pm1}]$.  Let
 $\bfa\in \Z^{2}$ and choose a nonsingular matrix
 $M\in \Z^{2\times2}$ and a vector $\bfb\in \Z^{2}$ with
 $\bfa=M\bfb$.  We have that
\begin{displaymath}
F(\bfx^{M})=x_{1}^{m_{1,1}}x_{2}^{m_{1,2}}+x_{1}^{m_{2,1}}x_{2}^{m_{2,2}}-2
\end{displaymath}
is irreducible, and so $P:=F(\bfx^{M})$ is the only irreducible factor
of this Laurent polynomial.  However, $t-1$ divides
$F(t^{\bfa})=P(t^{\bfb})$, and so these univariate Laurent polynomials
are not irreducible, unless we divide them by this cyclotomic factor.
\end{example}

Schinzel proved this conjecture when $n=1$ in \emph{loc. cit.} and,
under the restrictive hypothesis that $F$ is not self-inversive, when $n\ge2$
\cite{Schinzel:rlpI}, see also \cite[\S 6.2]{Schinzel:psrr}.  The
general case when $n\ge 2$ remains open. 

In Section~\ref{sec:fact-sparse-polyn}, we prove the function field
analogue for Laurent polynomials over the field $\C(z)$ in Theorem
\ref{thm:10} below. An element of $\C(z)[\bfx^{\pm 1}]$ is
\emph{constant} if it lies in $\C[\bfx^{\pm1}]$, up to a scalar factor
in $ \C(z)^{\times}$.  The \emph{constant part} of a Laurent
polynomial $F\in \C(z)[\bfx^{\pm1}]\setminus \{0\}$, denoted by
$\ct(F)$, is defined as its maximal constant factor. This constant
part is well-defined up to a unit of $\C(z)[\bfx^{\pm1}]$.

\begin{remark}
 \label{rem:3}
 The analogy between cyclotomic Laurent polynomials over $\Q$ and
 irreducible constant Laurent polynomials over $\C(z)$ stems from
 height theory. Let $\K$ denote either $\Q$ or $\C(z)$, and $\h$ the
 canonical height function on subvarieties of the torus
 $\G_{\rm m, \K}^{n}$, induced by the standard inclusion
 $\G_{\rm m, \K}^{n}\hookrightarrow \P^{n}_{\K}$.

 Let $F\in \K[\bfx^{\pm1}]$ be an irreducible Laurent polynomial
 defining a hypersurface $V(F)$ of $\Gm^{n}$. Then the condition that
 $\h(V(F))=0$ is equivalent to the fact that $F$ is cyclotomic when
 $\K=\Q$, and to the fact that $F$ is constant when $\K=\C(z)$.
\end{remark}

\begin{theorem}
\label{thm:10}
Let $F\in \C(z)[\bfx^{\pm1}]$ be a non-constant irreducible Laurent
polynomial. There are finite sets $\Omega^{0}\subset \Z^{n\times n}$
of nonsingular matrices and $\Gamma\subset \Z^n$ of nonzero vectors
satisfying the following property. Let $\bfa\in\Z^n$; then one of the next
conditions holds:
\begin{enumerate}
\item \label{item:14} there is $\bfc\in\Gamma$ verifying
 $\langle \bfc,\bfa\rangle=0$;
\item \label{item:15} there are $M\in\Omega^{0}$ and $\bfb\in\Z^n$
 with $\bfa=M\bfb$ such that if
$$
F(\bfx^M)=\prod_{P}P^{e_{P}}
$$ 
is the irreducible factorization of $F(\bfx^M)$, then 
\begin{displaymath}
\frac{F(t^{\bfa})}{\ct(F(t^{\bfa}))}=\prod_{P}\Big(\frac{P(t^{\bfb})}{\ct(P(t^{\bfb}))}\Big)^{e_{P}}
\end{displaymath}
is the irreducible factorization of $F(t^{\bfa})/\ct(F(t^{\bfa}))$.
\end{enumerate}
\end{theorem}

Similarly as for Conjecture \ref{Sch0}, for the validity this
statement it is is necessary to take out in its condition
\eqref{item:15} the constant part of $F(t^{\bfa})$ and of the
$P(t^{\bfb})$'s.

\begin{example}
 \label{exm:3}
Set $F=x_{1}+z x_{2}-z-1 \in \C(z)[x_{1}^{\pm1}, x_{2}^{\pm1}]$. Let $\bfa\in \Z^{2}$ and choose $M\in \Z^{2\times2}$ nonsingular and $\bfb\in \Z^{2}$ with $\bfa=M\bfb$. We have that
\begin{displaymath}
F(\bfx^{M})
=x_{1}^{m_{1,1}}x_{2}^{m_{1,2}}+z\, x_{1}^{m_{2,1}}x_{2}^{m_{2,2}}-z-1
\end{displaymath}
is irreducible, and so $P:=F(\bfx^{M})$ is its only irreducible
factor.  Again, $t-1$ divides $F(t^{\bfa})=P(t^{\bfb})$, and so these
univariate Laurent polynomials are not irreducible, unless we divide
them by a suitable constant factor.
\end{example}

\begin{remark}
  \label{rem:7}
  The validity of Schinzel's conjecture \ref{Sch0} would imply that of
  Conjecture \ref{conj:3}, in the same way that Theorem \ref{thm:10}
  implies Theorem \ref{thm:2}, as explained in Section
  \ref{sec:fact-sparse-polyn}.
\end{remark}



\section{A variant of Zannier's toric Bertini's theorem} 
\label{sec:toric-anal-bert}

{ Zannier proved in \cite{Zannier:hiaag} 
an analogue of Bertini's theorem for covers, where the subtori of $\Gm^{n}$ replaced the linear subspaces in the classical version of this theorem.}  This result was precised and generalized (with a completely different proof) by Fuchs,
Mantova and Zannier to include fibers of arbitrary cosets of subtori \cite[Theorem~1.5]{FuchsMantovaZannier:fiptvb} and to obtain a more
uniform result.

As before, let $\bfx=(x_{1},\dots,x_{n})$ be a set of $n$ variables
and denote by $\Gm^{n}=\Spec(\C[\bfx^{\pm1}])$ the $n$-dimensional
torus over $\C$. 

Let $W$ be a variety, that is, a reduced separated scheme of finite
type over $\C$. We assume that $W$ is irreducible and quasiprojective
of dimension $n\ge 0$, and equipped with a dominant (regular) map
\begin{equation*}
  \pi\colon W\longrightarrow \Gm^{n}
\end{equation*}
of degree $e\ge 1$ that is finite onto its image. Given an isogeny
$\lambda$ of $ \Gm^{n}$, that is, an endomorphism of $\Gm^{n}$ with
finite kernel, we denote by $\lambda^{*}W$ the fibered product
$\Gm^{n}\times_{\lambda,\pi} W$, and by
  \begin{equation}\label{eq:8}
    \xymatrix{
\lambda^{*}W \ar[r]^{\lambda} \ar[d]^{\pi}& W \ar[d]^{\pi}\\
\Gm^{n}\ar[r]^{\lambda} & \Gm^{n}}
  \end{equation}
  the corresponding fibered product square.

  \begin{definition}
    \label{def:1}
    The map $\pi$ satisfies the \emph{property PB (pullback)} if, for
    every isogeny $\lambda$ of $\Gm^{n}$, we have that $\lambda^{*}W$
    is an irreducible variety.
  \end{definition}

By \cite[Proposition 2.1]{Zannier:hiaag}, it is enough to test this
condition for $\lambda=[e]$, the multiplication map of $\Gm^{n}$ by
the integer $e=\deg(\pi)$.\\

The aforementioned result by Fuchs, Mantova and Zannier can be stated
as follows. 

\begin{theorem}  
\label{FMZ}
Let $W$ be an irreducible quasiprojective variety of dimension $n$ and
$\pi\colon W\rightarrow \Gm^n$ a dominant map that is finite onto its
image and that satisfies the property PB. There is a finite
union $\cE$ of proper subtori of $\Gm^n$ such that, for every subtorus
$T\subset \Gm^{n}$ not contained in $\cE$ and every point
$p\in \Gm^{n}(\C)= (\C^{\times})^{n}$, we have that
$\pi^{-1}(p\cdot T)$ is an irreducible subvariety of $W$.
\end{theorem}

%

When the property PB is not verified, the conclusion of this theorem
does not necessarily hold because the map $\pi$ factors through a
nontrivial isogeny, as it was already pointed out in
\cite{Zannier:hiaag}.

\begin{example}
\label{exm:4}
Let
$F=z^{2}-x_{1}x_{2}^{2}\in \C[x_{1}^{\pm1}, x_{2}^{\pm1},z^{\pm1}]$,
set $W$ be the torus $V(F)\subset \Gm^{3}$ and consider the isogeny
 \begin{displaymath}
\pi\colon W\longrightarrow \Gm^{2} 
 \end{displaymath}
 defined by $\pi(x_{1},x_{2},z)=(x_{1},x_{2})$ for
 $(x_{1},x_{2},z)\in W$.

The variety $W$ is irreducible and, since
 $F$ is monic in $z$, the map $\pi$ is finite. However, it does not
 satisfy the property PB, since for the isogeny $\lambda$ of
 $ \Gm^{2}$ defined by $\lambda (x_{1},x_{2})= (x_{1}^{2},x_{2})$,
\begin{displaymath}
\lambda^{*}W\simeq V(z-x_{1}x_{2}) \cup V(z+x_{1}x_{2})
\end{displaymath} 
and so this pullback is reducible.

Indeed, this map does neither satisfy the conclusion of
Theorem~\ref{FMZ}: given $(a_{1},a_{2})\in \Z^{2}\setminus \{(0,0)\}$
with $a_{1}$ even, let $T\subset \Gm^{2}$ be the 1-dimensional
subtorus given as the image of the map
$t\mapsto (t^{a_{1}},t^{a_{2}})$. Then
\begin{displaymath}
  \pi^{-1}(T)= V(z-t^{a_{1}/2}t^{a_{2}}) \cup
  V(z+t^{a_{1}/2}t^{a_{2}}), 
\end{displaymath}
and so this fiber is not irreducible.
\end{example}


Here we need a variant of Theorem~\ref{FMZ} that can be also applied
in the situation when the map $\pi$ does not verify the property PB.
In this more general situation, the conclusion of that theorem does
not necessarily hold. {\red However, as already remarked 
in~\cite[Proposition 2.1]{Zannier:hiaag}, we can reduce ourself,  
up to an isogeny, to a situation in which PB is satisfied. We need 
a more explicit statement. The conclusion of Theorem~\ref{FMZ}}
is replaced by an
alternative that ``explains'' the possibility that a fiber is
reducible by its factorization through a reducible pullback of the
variety $W$ by an isogeny of $\Gm^{n}$ within a finite set.

\begin{theorem}
\label{TBT}
Let $W$ be an irreducible quasiprojective variety of dimension $n$ and
$\pi\colon W\rightarrow \Gm^n$ a dominant map that is finite onto its
image. There is a finite union $\cE$ of proper subtori of $\Gm^n$ and
a finite set $\Lambda$ of isogenies of $\Gm^n$ such that, for each
subtorus $T\subset \Gm^{n}$ and each point
$p\in \Gm^{n}(\C)=(\C^{\times})^{n}$, one of the next conditions
holds:
\begin{enumerate}
\item \label{item:23}  $T\subseteq\cE$;
\item \label{item:24} there is $\lambda\in\Lambda$ with $\lambda^{*} W$
  reducible and a subtorus $T'\subset \Gm^{n}$ with $\lambda$ inducing
  an isomorphism $T'\to T$;
\item \label{item:25} $\pi^{-1}(p\cdot T)$ is irreducible. 
\end{enumerate}
\end{theorem}

\begin{remark}
  \label{rem:1}
  When the condition \eqref{item:24} above is satisfied, there is a
  diagram
  \begin{displaymath}
    \xymatrix{
\pi^{-1} (T')\ar[r] \ar[d]&
\lambda^{*}W\ar[r]^{\lambda} \ar[d]^{\pi}& W \ar[d]^{\pi}\\
T'\ar[r]^{\iota} & \Gm^{n}\ar[r]^{\lambda} & \Gm^{n}}
  \end{displaymath}
  with $\lambda^{*}W$ reducible and $\lambda \colon T'\to T$ an
  isomorphism, and where $\iota$ denotes the inclusion of the subtorus
  $T'$ into $\Gm^{n}$.

  Both inner squares in this diagram are fibered products, and so is
  the outer square. This implies that the fibers $ \pi^{-1}(T)$ and
  $ \pi^{-1}(T')$ are isomorphic. Thus $\pi^{-1}(T)$ can be identified
  with the fiber of a subtorus for the \emph{reducible} cover $\pi\colon\lambda^{*}W \to \Gm^{n}$, 
  and so this fiber is expected to be
  reducible as well.
\end{remark}

\begin{example}
  \label{exm:1}
  We keep the notation from Example \ref{exm:4}. In particular,
  $F=z^{2}-x_{1}x_{2}^{2}\in \C[x_{1}^{\pm1}, x_{2}^{\pm1},z^{\pm1}]$,
  $W$ the torus $V(F)\subset \Gm^{3}$, and
  $\pi\colon W\rightarrow \Gm^{2}$ the isogeny defined by
  $\pi(x_{1},x_{2},z)=(x_{1},x_{2})$.

  Let $(a_{1},a_{2})\in \Z^{2}\setminus \{(0,0)\}$ with $a_{1}$ even,
  and set $T\subset \Gm^{2}$ for the 1-dimensional
  subtorus given as the image of the map
  $t\mapsto (t^{a_{1}},t^{a_{2}})$.  These vectors satisfy the
  condition \eqref{item:24} in Theorem \ref{TBT} for the isogeny
  $\lambda\colon \Gm^{2}\to \Gm^{2}$ defined by
  \begin{displaymath}
    \lambda (x_{1},x_{2})=
    (x_{1}^{2},x_{2}).
  \end{displaymath}
  Indeed,  $\lambda^{*}W$ is reducible, and this isogeny
  induces an isomorphism $T'\to T$ with
  the subtorus $T'\subset \Gm^{2}$ given as the image of
  the map $t\mapsto (t^{a_{1}/2},t^{a_{2}})$.
\end{example}

We prove this theorem by reducing it to the previous toric Bertini's theorem, through a variation (Proposition \ref{red}) of a factorization result for rational maps from \cite{Zannier:hiaag}.
We give the proof 
after some auxiliary results.  We
first study the reducibility of pullbacks of varieties with respect
to isogenies of tori.

\begin{lemma} \label{lemm:2} Let $\pi\colon W\to X$ be a map of
  varieties and $\lambda\colon X\to X$ an \'etale map. Then
  $ X\times_{\lambda,\pi } W$ is a variety.

  In particular, for a map $\pi\colon W\to \Gm^{n}$ and an
  isogeny $\lambda $ of $\Gm^{n}$, we have that $\lambda^{*}W$ is a variety.
\end{lemma}

\begin{proof}
  Since  $\lambda\colon X\to X$ is  \'etale, the map 
  \begin{equation}\label{eq:7}
\lambda\colon X\times_{\lambda,\pi } W \longrightarrow W    
  \end{equation}
  is also \'etale, because of the invariance of this property under
  base change \cite[Chapter IV, Proposition 10.1(b)]{Hartshorne:ag}.
  By \cite[Chapter IV, Exercise 10.4]{Hartshorne:ag}, this implies
  that, for every closed point $q\in X\times_{\lambda,\pi } W$ and
  $p:=\lambda (q) \in W$, the induced map of completed local rings
  \begin{equation}
    \label{eq:3}
  \wh\cO_{p} \longrightarrow \wh\cO_{q}  
  \end{equation}
  is an isomorphism. Since $W$ is a variety, the local ring $\cO_{p}$
  is reduced and, by a theorem of Chevalley
  \cite[\S8.13]{ZariskiSamuel:caII}, the completion $ \wh\cO_{p}$ is
  reduced too.

  By the isomorphism in \eqref{eq:3}, the completed ring $ \wh\cO_{q}$
  is reduced. Since this is the completion of a ring with respect to a
  maximal ideal, the map $\cO_{q}\to \wh\cO_{q}$ is injective, and so
  the local ring $\cO_{q}$ is also reduced. Since the condition of
  being reduced is local, this implies that $X\times_{\lambda,\pi } W$
  is a variety.

  The last statement comes from the fact that the isogenies of
  algebraic groups over $\C$ are \'etale maps.
\end{proof}

Thanks to this result,  $\lambda^{*}W$ can be identified
with its underlying algebraic subset in the Cartesian product
$\Gm^{n}(\C)\times W(\C)$, namely
\begin{equation}\label{eq:5}
\lambda^{*}W= \{(p, w)\in \Gm^{n}(\C)\times W(\C) \mid \lambda(p)= \pi(w)\}. 
\end{equation}
Hence, $\lambda^{*}W$ is irreducible if and only if this algebraic
subset is irreducible. In particular, the map $\pi$ satisfies the
property PB if and only if for every isogeny $\lambda$ of $\Gm^{n}$,
the pullback $\lambda^{*}W$ has a single irreducible component.  

The following proposition is implicit in the proof of \cite[Proposition
2.1]{Zannier:hiaag}.

\begin{proposition}
\label{red}
Let $\pi\colon W\rightarrow \Gm^n$ be a map from an irreducible
variety $W$, and $\lambda$ an isogeny of $\Gm^{n}$. The following
conditions are equivalent:
\begin{enumerate}
\item \label{item:11}  the pullback $\lambda^{*}W$
is reducible;
\item \label{item:12} there is a factorization
  $\lambda = \mu\circ \tau$ with $\mu, \tau$ isogenies of $\Gm^{n}$
  such that $\mu$ is not an isomorphism, and a map
  $\rho\colon W\to \Gm^{n}$ such that $\pi=\mu\circ\rho$.
\end{enumerate}
\end{proposition}

In other terms, the condition \eqref{item:12} in the proposition above
amounts to the existence of the commutative diagram extending
\eqref{eq:8} of the form
  \begin{displaymath}
    \xymatrix{
\lambda^{*}W\ar[rr]^{\lambda} \ar[d]^{\pi}& &W \ar[d]_{\pi}\ar@/^4pc/[ddl]^{\rho}\\
 \Gm^{n}\ar[rr]^{\lambda} \ar[rd]_{\tau} & &\Gm^{n} \\
&\Gm^{n} \ar[ru]_{\mu}&
}
\end{displaymath}

\begin{proof}
  Suppose that the condition \eqref{item:12} holds. In this case, for
  $p\in \Gm^{n}(\C)$ and $w\in W(\C)$, the fact  that $ \lambda(p)=\pi(w)$
  is equivalent to $ \mu(\tau(p))=\mu(\rho(w))$, and so this holds if
  and only if there is $\zeta\in\ker(\mu)$ with
  $\tau (p)=\zeta\cdot\rho(w)$. From \eqref{eq:5}, the pullback
  decomposes into disjoints subvarieties as
  \begin{equation*}
\lambda^{*}W= \bigcup_{\zeta\in\ker(\mu)}
\Gm^{n}\times_{\tau,\zeta\cdot\rho} W. 
  \end{equation*}
  Since $\mu$ is not an isomorphism, this pullback is reducible,
  giving the condition \eqref{item:11}.

  Conversely, suppose that the condition \eqref{item:11} holds. Then
  $\lambda^{*}W$ has a decomposition into
  irreducible components
\begin{displaymath}
  \lambda^{*}W=\bigcup_{i=1}^{k}U_{i}
\end{displaymath}
with $k\ge 2$.  Similarly as in \eqref{eq:7}, the map
$ \lambda^{*}W\to W$ is \'etale, and so the $U_{i}$'s are
disjoint. Since $\lambda$ is an isogeny, the map $ \lambda^{*}W\to W$
is also finite.

The finite subgroup $\ker(\lambda)$ of $\Gm^{n}(\C)$ acts on
$\lambda^{*}W$ \emph{via} the maps $ (p,w)\mapsto (\zeta \cdot p,w) $
for $\zeta\in \ker(\lambda)$, and this action respects the fibers of
$\lambda$. The action is transitive on the fibers, and so it is also
transitive on the $U_{i}$'s.

Let $H\subset \ker(\lambda)$ be the stabilizer of the irreducible
component $U_{1}$, and $U_{1}/H$ the quotient variety.  We have that
$H$ acts on $U_{1}$ transitively on the fibers and without fixed
points. The induced map
\begin{displaymath}
U_{1}/H\longrightarrow W  
\end{displaymath}
is a finite \'etale map of degree 1, and so it is an isomorphism
\cite[\S III.10, Proposition 2]{Mumford:rbvs}.

Then we define the map $\rho\colon W\to \Gm^{n}$ as the map obtained
from the quotient map $U_{1}/H\to \Gm^{n}/H$ and the identifications
$U_{1}/H \simeq W$ and $\Gm^{n}/H\simeq \Gm^{n}$.  In concrete terms
and identifying $\Gm^{n}/H \simeq \Gm^{n}$, this map is defined, for
$w\in W$, as $\rho(w)= \tau(p\cdot H)$ for any $p\in \Gm^{n}$ such that
$(p,w)\in U_{1}$.

Both $\Gm^{n}/H$ and $ \Gm^{n}/\ker(\lambda )$ are
isomorphic to $\Gm^{n}$, and so there is a factorization
\begin{displaymath}
\lambda=\mu \circ \tau,   
\end{displaymath}
with $\tau$ and $\mu$ corresponding to the projections
$\Gm^{n}\to \Gm^{n}/H$ and $\Gm^{n}/H\to \Gm^{n}/\ker(\lambda)$,
respectively. For $w\in W$ and $(p,w)\in U_{1}$, we have that
$\mu\circ \rho (w)= \mu\circ \tau (p)= \pi(w)$. Since the action of
$\ker(\lambda)$ on the $U_{i}$'s is transitive and $k\ge 2$, we have
that $H\ne \ker(\lambda)$ and so $\mu$ is not an isomorphism, giving
the condition \eqref{item:12}.
\end{proof}

\begin{remark} 
\label{rem:6} 
By this proof, if $\lambda^{*}W$ is reducible, then the number of its
irreducible components is equal to the maximum of the quantity
$\deg(\mu)=[\ker(\lambda):H] $ over all possible maps $\rho$ as in the
condition \eqref{item:12}.
\end{remark}

The next result allows to factorize the dominant map
$\pi\colon W\to \Gm^{n}$ as a map satisfying the property PB followed
by an isogeny. It is a variant of
\cite[Proposition~2.1]{Zannier:hiaag}, that states a similar property
for dominant \emph{rational} maps.

\begin{corollary}
  \label{cor:2}
  Let $W$ be an irreducible variety of dimension $n$ and
  $\pi \colon W\to \Gm^{n}$ a dominant map. There are a map $\rho\colon W\to \Gm^{n}$
  satisfying the property PB  and an isogeny
  $\lambda$ of $\Gm^{n}$  with $\pi=\lambda\circ \rho$.
\end{corollary}

\begin{proof}
  Choose $\rho$ as a map $W\to \Gm^{n}$ of minimal degree among
  those that give a factorization of the form $\pi=\lambda \circ \rho$
  with $\lambda$ an isogeny of $\Gm^{n}$.

  Suppose that there is a further isogeny $\nu$ such that
  $\nu^{*}W=\Gm^{n}\times _{\nu,\rho} W$ is reducible. By Proposition
  \ref{red}, there would be an isogeny $\mu$ that is not an
  isomorphism and a map $\rho'\colon W\to \Gm^{n}$ with
  $\rho=\mu\circ \rho'$. Hence
  \begin{displaymath}
\pi= \lambda \circ \rho= (\lambda \circ\mu )\circ \rho' \and
\deg(\rho)=\# \ker(\mu) \cdot \deg(\rho')>\deg(\rho').
  \end{displaymath}
  By construction, this is not possible. Hence  $\nu^{*}W$
  is irreducible for every isogeny $\nu$ of $\Gm^{n}$, and so $\rho$
  satisfies the property PB.
\end{proof}



The next result gives a criterion to detect if the inclusion of a
subtorus can be factored through a given isogeny as in Proposition \ref{red}\eqref{item:12}.

\begin{lemma}
\label{split}
Let $T\subset \Gm^{n}$ be a subtorus and $\lambda$ an isogeny of
$\Gm^{n}$. The following conditions are equivalent:
\begin{enumerate}
\item \label{item:6} there is a subtorus $T'\subset \Gm^{n}$ such that
$\lambda$ induces an isomorphism $T'\to T$;
\item \label{item:7} $\lambda^{-1}(T)$ is the union of $\deg(\lambda)$
  distinct torsion cosets.
\end{enumerate}
\end{lemma}

\begin{proof}
  First suppose that $\lambda^{-1}(T)$ is the union of
  $\deg(\lambda)=\#\ker(\lambda)$ distinct torsion cosets, and denote
  by $T'$ the one that contains the neutral element. Then $T'$ is a
  subtorus and $T'\cap \ker(\lambda)=\{1\}$. It follows that
  $\lambda|_{T'} \colon T'\to T$ is an isogeny of degree 1 and hence an
  isomorphism, giving the first condition.

  Conversely, let  $T'\subset \Gm^{n}$ be a subtorus such that
  $\lambda|_{T'} \colon T'\to T$ is an isomorphism. Then 
  \begin{displaymath}
    \lambda^{-1}(T)= \ker(\lambda) \cdot T'.
  \end{displaymath}
  Since $T'\cap \ker(\lambda)=\{1\}$, this fiber is the union of
  $\#\ker(\lambda) = \deg(\lambda)$ distinct torsion cosets, giving
  the second condition.
\end{proof}

\begin{proof}[Proof of Theorem \ref{TBT}]
  By Corollary \ref{cor:2}, there are a map $\rho\colon W \to \Gm^{n}$
  satisfying the property PB and an isogeny $\lambda$ of $\Gm^{n}$
  with $\pi=\lambda\circ\rho$.  

  For each subgroup $H$ of $\ker(\lambda)$,  both
  $\Gm^{n}/H$ and $ \Gm^{n}/\ker(\lambda )$ are isomorphic to
  $\Gm^{n}$, and we consider then a factorization
\begin{equation} \label{eq:13}
\lambda=\mu_{H} \circ \tau_{H}   
\end{equation}
with $\tau_{H}$ and $\mu_{H}$ corresponding to the projections
$\Gm^{n}\to \Gm^{n}/H$ and $\Gm^{n}/H\to \Gm^{n}/\ker(\lambda)$,
respectively.  We set $\Lambda$ as the finite set of isogenies of
$\Gm^{n}$ of the form $\mu_{H}$ as above, for a proper subgroup $H$ of
$\ker(\lambda)$.

Since $\rho\colon W\rightarrow \Gm^{n}$ satisfies the property PB, by
\cite[Theorem 1.5]{FuchsMantovaZannier:fiptvb} there is a finite union
$\cE'$ of proper subtori of $\Gm^{n}$ such that, for every subtorus
$T$ of $\Gm^{n}$ not contained in $\cE'$ and every point
$p\in \Gm^{n}(\C)$, the fiber $\rho^{-1}(p\cdot T)$ is
irreducible. Set $\cE=\lambda(\cE')$.

We next show that the pair $(\Lambda, \cE)$ satisfies the requirements
of Theorem~\ref{TBT}.  Let $T$ be a subtorus of $\Gm^{n}$ that is not
contained in $\cE$ and write $\lambda^{-1}(T)=\bigcup_{i=1}^{k} T_{i}$
as a disjoint union of torsion cosets $T_{i}$ of $\Gm^{n}$.

When $k=1$, we have that $\lambda^{-1}(T)=T_1$
is a subtorus of $\Gm^{n}$ that is not contained in $\cE$. Hence,
$\pi^{-1}(T)=\rho^{-1}(T_1)$ is irreducible.

Otherwise, $k\ge 2$. Let $H \subset \ker(\lambda)$ be the stabilizer
of the (unique) subtori in this decomposition, say $T_{1}$. This is a
proper subgroup, because $\ker(\lambda)$ acts transitively on this
collection of torsion cosets and $k\ge 2$.

Consider the factorization $\lambda=\mu_H\circ\tau_H$ as in
\eqref{eq:13}. Then $\mu_H\in\Lambda$ and $\mu_H^{-1}(T)$ splits as an
union of $k=[\ker(\lambda):H] = \deg(\mu_{H})$ distinct torsion
cosets.  By Lemma \ref{split}, $\mu_{H}$ induces an isomorphism
between a subtorus $T'$ of $\Gm^{n}$ and $T$. Moreover,
Proposition~\ref{red}\eqref{item:12} applied to the map
$\tau_{H}\circ \rho$ and the isogeny $\mu_{H}$ shows that the pullback
$\mu_{H}^{*}W$ is reducible, completing the proof.
\end{proof}

It seems interesting to extend these results to maps that are not necessarily
dominant. In this direction, we propose the following conjectural
extension of the Fuchs-Mantova-Zannier theorem \ref{FMZ}. It can be seen as a
toric analogue of the classical Bertini's theorem as stated in
\cite[Th\'eor\`eme 6.3(3)]{Jou83}.

\begin{conjecture}  
\label{conj_FMZ}
Let $W$ be an irreducible quasiprojective variety and
$\varphi\colon W\rightarrow \Gm^n$ a map that is finite onto its image
and satisfies the property PB. There is a finite union $\cE$ of proper
subtori of $\Gm^n$ such that, for every subtorus $T$ of $\Gm^{n}$ with
\begin{displaymath}
  \dim(T)\ge  \codim(\ov{\varphi(W)})  +1
\end{displaymath}
that is not contained in $\cE$ and every point
$p\in \Gm^{n}(\C)$, we have that
$\varphi^{-1}(p\cdot T)$ is an irreducible subvariety of $W$.
\end{conjecture}

Similarly, we propose the  following conjectural extension of Theorem \ref{TBT}.

\begin{conjecture}
\label{conj_TBT}
Let $W$ be an irreducible quasiprojective variety and
$\varphi\colon W\rightarrow \Gm^n$ a map that is finite onto its
image. There is a finite union $\cE$ of proper subtori of $\Gm^n$ and
a finite set $\Lambda$ of isogenies of $\Gm^n$ such that, for each
subtorus $T\subset \Gm^{n}$ with
\begin{displaymath}
\dim(T)\ge \codim(\ov{\varphi(W)})  +1
\end{displaymath}
and each point
$p\in \Gm^{n}(\C)$, one of the next conditions
holds:
\begin{enumerate}
\item \label{item:23}  $T\subseteq\cE$;
\item \label{item:24} there is $\lambda\in\Lambda$ with $\lambda^{*} W$
  reducible and a subtorus $T'\subset \Gm^{n}$ with $\lambda$ inducing
  an isomorphism $T'\to T$;
\item \label{item:25} $\varphi^{-1}(p\cdot T)$ is irreducible. 
\end{enumerate}
\end{conjecture}

\section{Pullbacks of Laurent polynomials by monomial maps} \label{sec:pullb-laur-polyn}

 We next prove Theorem~\ref{BT-poly} stated in the introduction. To this
 end, we  first 
recall some notation and introduce some auxiliary results.

Let $\bft =(t_1,\ldots,t_{k})$ be a set of $k$ variables. A 
matrix $A=(a_{i,j})_{i,j}\in\Z^{n\times k}$ defines  the family of
$n$ monomials in the variables $\bft$ given by
$$
\bft^A=\Big(\prod_{j=1}^{k} t_j^{a_{1,j}},\ldots,\prod_{j=1}^k t_j^{a_{n,j}}\Big).
$$
The rule $\bft\mapsto \bft^{A}$ defines a $k$-parameter monomial map
$\Gm^{k}\rightarrow \Gm^{n} $. This is a group morphism and indeed,
every group morphism from $\Gm^{k}$ to $\Gm^{n}$ is of this form. The
isogenies of $\Gm^{n}$ correspond to the nonsingular matrices of
$\Z ^{n\times n}$.

Given $\bfa=(a_{1},\dots, a_{n})\in \Z^{n}$, we can consider it as a
row vector, that is, as a matrix in $\Z^{1\times n}$. In this case,
\begin{displaymath}
  \bfx^{\bfa}= \prod_{j=1}^{n}x_{j}^{a_{j}}
\end{displaymath}
is an $n$-variate monomial.  Row vectors give characters of $\Gm^{n}$,
that is, group morphisms $\Gm^{n}\to \Gm$. When  $\bfa$ is
primitive, the kernel of its associated character is a subtorus of
$\Gm^{n}$ of codimension 1, and every such subtorus arises in this way.

Else, we can  consider $\bfa$ as a column vector, that is, as a
matrix in $\Z^{n\times 1}$. Then
\begin{displaymath}
  t^{\bfa}=(t^{a_{1}},\dots, t^{a_{n}})
\end{displaymath}
is a collection of $n$ univariate monomials in a variable $t$. Column
vectors give group morphisms $\Gm\to \Gm^{n}$. When $\bfa\ne 0$, the
image of such a morphims is a subtorus of $\Gm^{n}$ of dimension 1,
that we denote by $T_{\bfa}$. When $\bfa$ is primitive, the associated
group morphism $\Gm\to \Gm^{n}$ gives an isomorphism between $\Gm$ and
$T_{\bfa}$.

For subvarieties of tori, fibered products like those in \eqref{eq:8}
can be expressed in more concrete terms. The next lemma gives such an
expression for the case of hypersurfaces.

\begin{lemma} \label{lemm:3} Let $F \in \C[\bfx^{\pm 1},z^{\pm1}]$,
  $G\in \C[\bfx^{\pm1}]\setminus \{0\}$, and $A\in \Z^{n\times k}$.
  Let $W$ be the hypersurface of $\Gm^{n+1}\setminus V(G)$ defined by
  $F$, $\pi\colon W\to \Gm^{n}$ the map defined by
  $\pi(\bfx,z)= \bfx$, and $\lambda \colon \Gm^{k}\to \Gm^{n}$ the
  group morphism defined by $\lambda(\bft)= \bft^{A}$.  Then
  $\Gm^{k}\times_{\lambda,\pi} W $ is isomorphic to the subscheme of
  $\Gm^{k+1}\setminus V(G(\bft^{A})) $ defined by $F(\bft^{A},z)$.
\end{lemma}

\begin{proof}
The maps $\pi $ and $\lambda$ correspond to the morphisms of
$\C$-algebras
\begin{displaymath}
  \C[\bfx^{\pm1}]\longrightarrow \C[\bfx^{\pm1},z^{\pm1}]_{G}/F \and
  \C[\bfx^{\pm1}]\longrightarrow \C[\bft^{\pm1}]\simeq \C[\bfx^{\pm1},\bft^{\pm1}]/(\bfx-\bft^{A}),
\end{displaymath}
and the fibered product $\Gm^{k}\times_{\lambda,\pi} W $ is the
scheme associated to the tensor product
\begin{displaymath}
\C[\bfx^{\pm1},z^{\pm1}]_{G}/F \otimes_{ \C[\bfx^{\pm1}]}
 \C[\bfx^{\pm1},\bft^{\pm1}]/(\bfx-\bft^{A})  .
\end{displaymath}
This tensor product is isomorphic to the $\C$-algebra
\begin{displaymath}
  \C[\bfx^{\pm1},z^{\pm1},\bft^{\pm1}]_{G}/(F,\bfx-\bft^{A})  
\simeq 
\C[z^{\pm1},\bft^{\pm1}]_{G(\bft^{A})}/(F(\bft^{A},z)),
\end{displaymath}
which gives the statement.
\end{proof}

\begin{lemma} \label{lemm:1} Let $f\in\C(t)[z]$ be an irreducible
  polynomial of degree $d\geq1$, and such that $f(t^m,z)$ is reducible
  for some $m\in\N$. There is $e\in\N$ dividing $\gcd(m,d)$ such
  that $f(t^e,z)$ is also reducible.
\end{lemma}

\begin{proof}
The proof relies on the action of torsion points on irreducible factors as in 
\cite[Proposition 2.1]{Zannier:hiaag}.

  By Lemma \ref{lemm:3}, the subscheme of $\Gm^{2}$ defined by
  $f(t^{m},z)$ is isomorphic to the pullback $[m]^{*} V(f)$, with
  $[m]$ the $m$-th multiplication map of $\Gm$. By Lemma \ref{lemm:2},
  this pullback is reduced, and so $f(t^{m},z)$ is
  separable. Consider its decomposition into distinct irreducible
  factors 
\begin{equation}\label{eq:15}
  f(t^{m},z)= \prod_{i=1}^{k} p_{i},
\end{equation}
with
$k\ge 2$.

The group $\upmu_{m}$ of $m$-th roots of the unity acts on the set of
these irreducible factors by
$p_{i}(t,z)\mapsto p_{i}(\zeta\cdot t,z)$, $i=1,\dots, k$, for
$\zeta\in \upmu_{m}$.  Let $\cP\subset \{p_{1},\dots, p_{k}\}$ be a
nonempty orbit of this action. The polynomial
\begin{equation*}
 \prod_{p\in \cP} p 
\end{equation*}
is invariant under the action of $\upmu_{m}$, and so it is of the form
$g(t^{m},z)$ with $g\in \C(t)[z]$.  This product is a nontrivial
factor of $f(t^{m},z)$, and so $g$ coincides with $f$ up to a
scalar. It follows that $\cP= \{p_{1},\dots, p_{k}\}$ and so the action
is transitive. In particular, all the $p_{i}$'s have the same degree
in the variable $z$, and so this degree is positive and $k | d$.

The stabilizer of an irreducible factor $p_{i}$ is a subgroup of
$\upmu_{m}$, hence it is of the form $\upmu_{l}$ with $l | m$.  Since
the action is transitive and $\upmu_m$ is abelian, this subgroup does
not depend on the choice of $p_{i}$.  
Moreover, $m/l$ is equal to $k$,
the number of irreducible factors of $f(t^{m},z)$, also because of the
transitivity of the action.

By the invariance of each  $p_{i}$  under the action of $\upmu_{l}$, 
there is $q_{i}\in \C(t) [z] \setminus \C(t)$ with
$p_{i}=q_{i}(t^{l},z)$. It follows from \eqref{eq:15} that
\begin{displaymath}
  f(t^{e},z) =  \prod_{i=1}^{k} q_{i}(t,z),
\end{displaymath}
with $e=m/l$. Clearly $e|m$ and as explained, $e=k$, and so this
quantity also divides~$d$, completing the proof.
\end{proof}

\begin{lemma}
\label{specialization}
Let $F\in\C[\bfx,z] $ be an irreducible polynomial of
degree $d\geq1$ in the variable $z$, and
$G\in \C[\bfx]\setminus \{0\}$ its leading coefficient.

\begin{enumerate}
\item \label{item:13} Let $W=V(F)\setminus V(G) \subset \Gm^{n+1}$ and
  $\pi\colon W\rightarrow \Gm^{n}$ the map defined by
  $\pi(\bfx,z)=\bfx$. The image of $\pi$ is the open subset
  $\Gm^{n}\setminus V(G)$ of $\Gm^{n}$, and this map is finite onto
  this open subset.

\item \label{item:16} There is a finite subset $\Delta_F$ of $\Z^n$
  such that for $\ag\in\Z^n$ with $\langle \cg,\ag\rangle\neq0$ for
  all $\cg\in\Delta_F$, the polynomial $F(t^\ag,z)$ has degree $d$ in
  the variable $z$.

\item \label{item:19} If $A\in \Z^{n\times n}$ is nonsingular, then
  $F(\bfx^{A},z)$ has no nontrivial factors in $\C[\bfx^{\pm1}]_{G}$.
\end{enumerate}
\end{lemma}

\begin{proof}
  For the first statement, the image of the map $\pi$ is contained in
  the open set $U=\Gm^{n}\setminus V(G)$. The induced map $W\to U$
  corresponds to the morphism of $\C$-algebras
\begin{displaymath}
  \C[\bfx^{\pm 1}]_{G} \longhookrightarrow \C[\bfx^{\pm1},z]_{G}/(F).
\end{displaymath}
This morphism is an integral extension because the leading term $G$ is
invertible in $\C[\bfx^{\pm 1}]_{G}$, and so the map $W\to U$ is
finite and, \emph{a fortiori}, surjective.

For the second statement, write
$G=\sum_{j=1}^{r} G_{j}\bfx^{\bfc_{j}}$ with $G_{j}\in \C^{\times}$
and $\bfc_{j}\in \N^{n}$, $j=1,\dots, r$, and consider the finite
subset of $\Z^{n}$ given by 
\begin{displaymath}
  \Delta_{F}=\{\bfc_{j}-\bfc_{1} \mid j=2,\dots, r\} .
\end{displaymath}
For $\bfa\in\Z^{n}$ with $\langle \cg,\ag\rangle\neq0$ for all
$\cg\in\Delta_F$, we have that $G(t^{\bfa})\ne 0$ and so
$\deg_{z}(F(t^{\bfa},z))=d$.

As for Lemma \ref{lemm:1}, the proof of the last assertion relies on the action of torsion points on irreducible factors, and so we only sketch it. Using Lemmas \ref{lemm:3} and
\ref{lemm:2}, we show that  $ F(\bfx^{A}, z)$ is
separable.  Let
\begin{displaymath}
  F(\bfx^{A},z)=\prod_{i=1}^{k}P_{i}
\end{displaymath}
the decomposition of this Laurent polynomial into distinct irreducible
factors.  The action of the finite group
$\{\bfx \in \Gm^{n}\mid \bfx^{A}=1\}$ on the the sets of these
irreducible factors is transitive, and so the $P_{i}$'s have the same
degree with respect to the variable $z$. Hence for $i=1,\dots, k$, we
have that $k \deg_{z}(P_{i}) = d \ge 1$.  In particular,
$\deg_{z}(P_{i})\ge 1$, proving the statement.
\end{proof}




\begin{proof}[Proof of Theorem~\ref{BT-poly}]
  The statement of this result, restricted to \emph{primitive} vectors  $\bfa\in \Z^{n}$,
  is a specialization of Theorem~\ref{TBT}. To see this, first
  reduce, multiplying by a suitable monomial, to the case when $F$ is
  an irreducible polynomial in $\C[\bfx,z ]$ of degree $d\ge 1$ in the
  variable $z$.  Set $W=V(F)\setminus V(G)$ and consider the map
\begin{equation*} 
\pi\colon   W\longrightarrow \Gm^{n}
\end{equation*}
defined by $\pi(\bfx,z)=\bfx$ for $(\bfx,z)\in W$. The
quasi-projective variety $W$ is irreducible and, by Lemma
\ref{specialization}\eqref{item:13}, this map is dominant and finite
onto its image, the open subset $U=\Gm^{n}\setminus V(G)$ of
$\Gm^{n}$. 

Let $\Lambda$ be a finite subset of isogenies of $\Gm^{n}$ and $\cE$ a
finite union of proper subtori of $\Gm^{n}$ satisfying the conclusion
of Theorem~\ref{TBT} applied to this map. Set then
$\Phi_{1}$ for the finite subset of nonsingular matrices in
$ \Z^{ n\times n} $ corresponding to the isogenies in $\Lambda$, and
$\Sigma_{1} $ for a finite subset of nonzero vectors of $\Z^{n}$ such that
\begin{equation}
  \label{eq:10}
 \cE\subset \bigcup_{\bfc\in \Sigma_{1}} V(\bfx^{\bfc}-1).
\end{equation}

For a primitive vector $\bfa\in \Z^{n}$, set $T_{\bfa}$ for the
1-dimensional subtorus defined as the image of the group morphism
$\Gm\to \Gm^{n}$. This map gives an isomorphism between $\Gm$ and
$T_{\bfa}$. By Lemma \ref{lemm:3}, the fiber $\pi^{-1}(T_{\bfa})$ is
isomorphic to the subscheme of $\Gm^{2} \setminus V(G(t^{\bfa}))$
defined by $F(t^{\bfa},z)$.  For the isogeny $\lambda$ associated to a
nonsingular matrix $M\in \Phi_{1}$, the same result shows that
$\lambda^{*}W$ is isomorphic to the subscheme of
$\Gm^{n+1}\setminus V(G) $ defined by $F(\bfx^{M},z)$.

The three alternatives from Theorem~\ref{TBT} applied to the map
$\pi$, the subtorus $T_{\bfa}$ and the point
$p=(1,\dots,1) \in \Gm^{n}(\C)$, then boil down to those in the
theorem under examination, as explained below.
\begin{enumerate}
\item \label{item:21} Suppose that $T_{\bfa}\subset \cE$. By
  \eqref{eq:10}, there is $\bfc\in \Sigma_{1}$ with
  $\langle \bfc,\bfa\rangle =0 $.
\item \label{item:22} Else suppose that there is an isogeny
  $\lambda \in \Lambda$ with $\lambda^{*}W$ reducible and a subtorus
  $T'$ of $\Gm^{n}$ with $\lambda$ inducing an isomorphism between
  $T'$ and $T_{\bfa}$. For $M\in \Phi_{1}$  the
  nonsingular matrix associated to $\lambda$, we have that  $\bfa\in \im(B)$
  and, by Lemma~\ref{lemm:3}, $F(\bfx^{M},z)$ is reducible.
\item\label{item:20} Else suppose that $\pi^{-1}(T_{\bfa})$ is
  irreducible in $\Gm^{2}\setminus V(G)$.  By Lemma~\ref{lemm:3}, this
  implies that $F(t^{\bfa},z)$ is irreducible in $\C(t)[ z^{\pm1}]$.
\end{enumerate}

We next enlarge these finite sets to cover the rest of the cases.  Let
$d\ge 1$ be the degree of $F$ in the variable $z$, and let $ e$ be a
divisor of $d$. If $F(\bfx^{e},z)$ is irreducible, we respectively
denote by $\Phi_{e}$ and $\Sigma_{e}$ the finite subsets of
nonsingular matrices in $ \Z^{ n\times n} $ and of nonzero vectors of
$\Z^{n}$ given by the application of Theorem~\ref{TBT} to this
polynomial.  Otherwise, we set $\Phi_{e} =\{I_{n}\}$ with $I_{n}$ the
identity matrix of $\Z^{n\times n}$, and $\Sigma_{e} =\emptyset$.  Set
also $\Delta $ for the finite subset of nonzero vectors in $\Z^{n}$
associated to $F$ by Lemma \ref{specialization}\eqref{item:16}.  Set
then
\begin{displaymath}
  \Phi = \bigcup_{e\mid d} e \, \Phi_{e} \and \Sigma= \Delta\cup \bigcup_{e\mid d}
  \Sigma_{e}. 
\end{displaymath}
By Theorem \ref{TBT} and the previous analysis, the statement holds
for all vectors of the form $e\, \bfb$ with $\bfb\in \Z^{n}$ primitive
and $e |d$. 

Given an arbitrary vector $\bfa\in \Z^{n}$, write $ \bfa= m\bfb$ with
$m\in \N$ and $\bfb \in \Z^{n}$ primitive, and set
\begin{displaymath}
  f=F(t^{\bfb},z)\in \C[t^{\pm1},z].
\end{displaymath}
Suppose that neither \eqref{item:8} nor \eqref{item:9} hold for
$\bfa$. Let $e\in \N$ be a common divisor of $d$ and $m$.  \emph{A
  fortiori}, these conditions do neither hold for $e\, \bfb$ and, as
explained before,
\begin{displaymath}
  f(t^{e},z)= F(t^{e\bfb},z)
\end{displaymath}
is irreducible in $\C(t)[z]$. By Lemma \ref{lemm:1}, we have that
$ F_{\bfa}=f(t^{m},z)$ is irreducible in $\C(t)[z]$, giving the
condition \eqref{item:10} for $\bfa$ and concluding the proof.
\end{proof}


\begin{remark} 
  \label{rem:4}
Using the toric  Bertini's theorem \ref{TBT} for cosets of arbitrary dimension, the
present polynomial version in Theorem \ref{BT-poly} might be extended
to $k$-parameter monomial maps for any $k$, and also to arbitrary
  translates of these monomial maps.  

  We have kept the present more restricted statement for the sake of
  simplicity, and also because it is sufficient for our application.
\end{remark}

\section{Factorization of sparse polynomials} 
\label{sec:fact-sparse-polyn}

Here we prove the results on the factorization of Laurent polynomials
announced in the introduction and in Section~\ref{sec:conj-sche-funct}. 
Theorem~\ref{thm:10} is easily seen to be implied by the following
statement. 

\begin{theorem}
\label{thm:1}
Let $F\in \C[\bfx^{\pm1},z^{\pm1}]$ without nontrivial factors in
$\C[\bfx^{\pm1}]$.  There are finite sets
$\Omega^{0}\subset \Z^{n\times n}$ of nonsingular matrices and
$\Gamma\subset \Z^n$ of nonzero vectors satisfying the property  that, for
$\bfa\in\Z^n\setminus \{\bfzero\}$, one of the next alternatives
holds:
\begin{enumerate}
\item \label{item:17} there is $\bfc\in\Gamma$ with
  $\langle \bfc,\bfa\rangle=0$;
\item \label{item:18} there are $M\in\Omega^{0}$ and $\bfb\in\Z^n$ with
  $\bfa=M\bfb$ such that if 
$$
F(\bfx^M,z)=\prod_{P}P^{e_{P}}
$$ 
is the irreducible factorization of $F(\bfx^M,z)$ in
$\C[\bfx^{\pm1},z^{\pm1}]$, then 
\begin{displaymath}
F(t^{\bfa},z)=\prod_{P}P(t^{\bfb},z)^{e_{P}}
\end{displaymath}
is the irreducible factorization of $F(t^{\bfa},z)$ in
$ \C(t)[z^{\pm1}]$.
\end{enumerate}
\end{theorem}

\begin{proof}
  We proceed by induction on $\deg_{z}(F)$.  When $\deg_{z}(F)=0$, the
  statement is trivial, and so we assume that $\deg_{z}(F)\ge 1$.

If $F$ is irreducible, we respectively
  denote by $\Phi $ and $\Sigma $ the finite sets of nonsingular
  matrices in $\Z^{n\times n}$ and of nonzero vectors in $\Z^{n}$ from
  Theorem~\ref{thm:1} applied to this Laurent polynomial. If $F$ is
  reducible, we set $\Phi=\{I_{n}\}$ and $\Sigma=\emptyset$.

  Let $\ag\in\Z^n$.  When $F$ is irreducible, if the condition
  \eqref{item:8} in Theorem~\ref{BT-poly} holds, then the condition
  \eqref{item:17} in Theorem~\ref{thm:1} also holds by taking $\Gamma$
  as any finite set containing~$\Sigma $. Still in the irreducible
  case, if the condition \eqref{item:10} in Theorem~\ref{BT-poly}
  holds, the Laurent polynomial $F(t^\ag,z)$ is irreducible, and the
  condition \eqref{item:17} in Theorem~\ref{thm:1} holds provided that
  $\Omega^{0}$ contains $I_{n}$.

Else, suppose that the condition  \eqref{item:9} in
Theorem~\ref{BT-poly} holds, that is, there are
$M\in\Phi $ and $\bg\in\Z^n$ with  $\ag=M\bg$ and
$F(\xg^M,z)$ is reducible. Let
\begin{equation*}
F(\xg^M, z)=F_1 \, F_2
\end{equation*}
be a nontrivial factorization. By
Lemma~\ref{specialization}\eqref{item:19}, $F(\xg^M,z)$ has no factors
in $\C[\bfx^{\pm1}]$. Hence $\deg_{z}(F_{i})<\deg_{z}(F)$, $i=1,2$,
and by induction, Theorem~\ref{thm:1} holds for these Laurent
polynomials.  Let $\Omega^{0}_{i}$ and $\Gamma_{i}$ respectively denote
the finite sets of nonsingular matrices in $\Z^{n\times n}$ and of
nonzero vectors in $\Z^{n}$ whose existence is assured by this
theorem.

By construction, either  there is a vector
$\cg\in\Gamma_{1}\cup\Gamma_{2}$ with 
$\langle\cg,\bg\rangle=0$, or we can find $M_i\in\Omega_i$ and
$\bg_i\in\Z^n$ with  $\bg=M_i\bg_i$ and a decomposition
$$
F_i(\xg^{M_i},z)=\prod_{j=1}^{k_{i}}F_{i,j}
$$
with $F_{i,j}(t^{\bg_i},z)$ irreducible in $\C(t)[s^{\pm1}]$ for all
$i,j$.

Set
\begin{equation}
  \label{eq:14}
\Gamma=\{\Adj(M) \bfc \mid M\in \Phi,  \cg\in
\Gamma_{1}\cup\Gamma_{2}\}
\end{equation}
with $\Adj(M)$ the adjoint matrix of $M$.  If
$\langle\cg',\ag\rangle\neq0$ for all $\cg'\in\Gamma$, then
$\langle\cg,\bg\rangle \ne 0$ for all
$\cg\in \Gamma_{1}\cup\Gamma_{2}$ and so $F_{i,j}(t^{\bg_i},z)$ is
irreducible in $\C(t)[s^{\pm1}]$ for all $i,j$.

Consider the lattices $K_i=\Im(M_i)$, $i=1,2$, and set
$K=K_1\cap K_2$. Since $K$ is also a lattice, there is a nonsingular
matrix $M'\in\Z^{n \times n}$ with $K=\Im(M')$ and, since
$K\subseteq K_i$, there are nonsingular matrices $N_i$, $i=1,2$, with
$M'=M_iN_i$.  Furthermore, $\bg\in K$ implies that there is
$\bg'\in\Z^n$ with $\bg=M'\bg'=M_iN_i\bg'$. Hence
$ \bg_i=M_i^{-1}\bg=N_i\bg' $ and
\begin{equation*}
  F(\xg^{MM'},z)
 =F_1 (\xg^{M_1N_1},z)F_2(\xg^{M_2N_2},z )
=\prod_{i=1}^2\prod_{j=1}^{r_i} F_{i,j} (\xg^{N_i},z ).
\end{equation*}
Set $M''=MM'$, $G_{i,j}=F_{i,j} (\xg^{B_i},z )$ and consider the decomposition 
$$
F (\xg^{M''},z )=\prod_{i=1}^2\prod_{j=1}^{r_i} G_{i,j}. 
$$ 
We have $\ag=M\bg=M''\bg'$ and 
$$
G_{i,j}(t^{\bg'},z)=F_{i,j} (t^{B_i\bg'},z )=F_{i,j} (t^{\bg_i} ,z)
$$ 
is irreducible in $\C(t)[s^{\pm1}]$ for all $i,j$. The statement
follows by taking $\Omega^{0}$ as any finite set containing all the
matrices of the form $MM'$ for $M\in \Phi$, and $\Gamma$ as in
\eqref{eq:14}.
\end{proof}

We conclude by giving the proof of our main result.

\begin{proof}[Proof of Theorem \ref{thm:2}]
  We proceed by induction on $n$. When $n=0$ the statement is trivial,
  and so we assume that $n\ge 1$.  Let $F\in \C[\bfx^{\pm1},z^{\pm1}]$
  and write
  \begin{displaymath}
    F=CF'
  \end{displaymath}
  with $C\in \C[\bfx^{\pm1}]$ and $F'\in \C[\bfx^{\pm1},z^{\pm1}]$
  without nontrivial factors in $\C[\bfx^{\pm1}]$.  By Lemma
  \ref{specialization}\eqref{item:16}, there is a finite subset
  $\Delta\subset \Z^{n}$ such that $C(t^{\bfb})\ne 0$ for all
  $\bfb\in \Z^{n}$ with $\langle \bfc,\bfb\rangle\ne 0$ for all
  $\bfc\in \Delta$.  Let also $\Omega^{0}\subset \Z^{n\times n}$ and
  $\Gamma\in \Z^{n}$ be the finite subsets given by Theorem
  \ref{thm:1} applied to $F'$.

  Let $\bfa\in \Z^{n}$.  When $\langle \bfc,\bfa\rangle\ne 0$ for all
  $\bfc\in \Gamma\cup \Delta$, Theorem \ref{thm:1}\eqref{item:18}
  implies the statement, provided that we choose any finite subset
  $ \Omega\subset \Z^{n\times n}$ containing $\Omega^{0}$.

  Otherwise, suppose that there is $\bfc\in \Gamma\cup \Delta$ with
  $\langle \bfc,\bfa\rangle = 0$.  If $C(t^{\bfa},z)=0$, we add to the
  finite set $\Omega$ the matrix $M\in \Z^{n\times n}$ made by adding
  to $n-1$ zero columns to the vector $\bfa$.  Otherwise, choose a
  matrix $L\in \Z^{n\times (n-1)}$ defining a linear map
  $ \Z^{n-1}\rightarrow \Z^{n}$ whose image is the submodule
  $c^{\bot}\cap \Z^{n}$, and a vector $\bfd\in \Z^{n-1}$ with
  $\bfa=L\bfd$.  Let $\bfu=(u_{1},\dots, u_{n-1})$ be a set of $n-1$
  variables and set
  \begin{displaymath}
    G=F'(\bfu^{L})\in \C[\bfu^{\pm1},z^{\pm1}]  .
  \end{displaymath}

  By the inductive hypothesis, there is a finite subset
  $\Omega_{\bfc}\subset\Z^{(n-1)\times (n-1)}$ satisfying the
  statement of Theorem \ref{thm:2} applied to this Laurent
  polynomial. In particular, there are $N\in \Omega_{\bfc}$ and
  $\bfe\in \Z^{n-1}$ with $\bfd=N\bfe $ such that, for an irreducible
  factor $Q$ of $G(\bfu^{N},z)$, we have that $Q(t^{\bfe},z)$ is, as a
  Laurent polynomial in $\C(t)[z^{\pm1}]$, either a unit or an
  irreducible factor of $G( t^{\bfd},z)$.

  We have that $G(\bfu^{N},z)=F'(\bfu^{LN},z)$, and so $Q$ is an
  irreducible factor of this latter Laurent polynomial. Moreover,
  $\bfa=LN \bfe$.  Enlarging the matrix $LN\in \Z^{n\times (n-1)}$ to
  a matrix $M\in \Z^{n\times n}$ by adding to it a zero column at the
  end, and similarly enlarging the vector $\bfe$ to a vector
  $\bfb\in \Z^{n}$ by adding to it a zero entry at the end, the
  previous equalities are preserved with $M$ and $\bfb$ in the place
  of $LN$ and $\bfe$. Hence, $ \bfa=M\bfb $ and, if
  $Q(x_{1},\dots, x_{n-1})$ is an irreducible factor of
  $F'(\bfx^{M},z)$, then $Q(t^{\bfe},z)=Q(t^{\bfb},z)$ is, as a
  Laurent polynomial in $\C(t)[z^{\pm1}]$, either a unit or an
  irreducible factor of $G( t^{\bfd},z)=F(t^{\bfa},z)$.

  The statement then follows by also also adding to $\Omega$ all the
  matrices $M\in \Z^{n\times n}$ constructed in this way.
\end{proof}

\begin{remark}
  \label{rem:10}
  In the setting of Theorem \ref{thm:2}, the bivariate Laurent
  polynomials $F_{\bfa}$ can be defined as the pullback of the
  multivariate Laurent polynomial $F$ by the 2-parameter monomial map
  $(t,z)\mapsto (t,z)^{A}$ given by the matrix
\begin{displaymath}
  A=
  \begin{pmatrix}
    0 &1\\
a_{1} & 0\\
\vdots &\vdots\\
a_{n} & 0
  \end{pmatrix} \in \Z^{(n+1)\times 2}.
\end{displaymath}
In Conjecture \ref{conj:3} applied to $F$ and $k=2$, one can
consider \emph{all} matrices in $ \Z^{(n+1)\times 2}$, and so its setting is
more general than that of Theorem \ref{thm:2}. 

On the other hand, the conclusion of Conjecture \ref{conj:3} in this
situation is slightly weaker than that of Theorem \ref{thm:2}, since
it does not give the irreducible factorization of the $F_{\bfa}$ in
$\C(t)[z^{\pm1}]$, but rather its irreducible factorization modulo the
Laurent polynomials of the form $f(t^{d_{1}}z^{d_{2}})$ for a
univariate $f$ and $d_{1},d_{2}\in \Z$.
\end{remark}

\bibliographystyle{amsalpha}
\bibliography{biblio}

\end{document}